\documentclass[11pt,reqno]{amsart}
\usepackage[utf8]{inputenc} 
\usepackage[T1]{fontenc}    
\usepackage{hyperref}
\usepackage{amsfonts}       
\usepackage{microtype}      

\usepackage{amscd,amssymb,amsmath,amsthm}
\usepackage[arrow,matrix]{xy}
\usepackage{graphicx}
\usepackage{multicol}
\usepackage{subfigure}
\usepackage{epstopdf}
\usepackage{color}
\newtheorem{defn}{Definition}
\newtheorem{lemma}{Lemma}
\newtheorem{pro}{Proposition}
\newtheorem{rk}{Remark}

\numberwithin{equation}{section} 

\begin{document}

\vspace{0.5in}
\renewcommand{\bf}{\bfseries}
\renewcommand{\sc}{\scshape}
\vspace{0.5in}

\title[Discrete dynamics of two phytoplankton-zooplankton systems]{Discrete dynamics of phytoplankton-zooplankton systems with Holling type functional responses}

\author{S.K. Shoyimardonov, U.A. Rozikov}

\address{S.K. Shoyimardonov$^{a,b}$ \begin{itemize}
\item[$^a$] V.I.Romanovskiy Institute of Mathematics, 9, University str.,
Tashkent, 100174, Uzbekistan;
\item[$^b$] National University of Uzbekistan,  4, University str., 100174, Tashkent, Uzbekistan.
\end{itemize}}
\email{shoyimardonov@inbox.ru}

\address{ U.A. Rozikov$^{a,b,c}$\begin{itemize}
		\item[$^a$] V.I.Romanovskiy Institute of Mathematics,  9, Universitet str., 100174, Tashkent, Uzbekistan;
		\item[$^b$]  National University of Uzbekistan,  4, Universitet str., 100174, Tashkent, Uzbekistan;
		\item[$^c$] Karshi State University, 17, Karabag str., 180100, Karshi,  Uzbekistan.
\end{itemize}}
\email{rozikovu@yandex.ru}







\begin{abstract} In this paper, we study discrete dynamics of phytoplankton-zooplankton system with Holling type II and Holling type III predator functional responses. The stability of positive fixed points are investigated. By finding invariant sets, the convergence of the trajectory was shown under certain conditions on the parameters.
\end{abstract}

\maketitle

\section{Introduction}

The study of ecosystem dynamics is always an important and interesting research object in the theory of dynamic systems.  Phytoplankton is very important as it contributes almost half of the photosynthesis flow on our planet and absorbs a third of the carbon dioxide.  Zooplankton feed on toxins, phytoplankton, and they form the base of the ocean food chain.  Thus, phytoplankton-zooplankton interactions have been studied by many researchers based on various models (\cite{Chatt, Chen, Hong, Tian, RSH, RSHV, SH, Qiu}) and references therein. Although the study of discrete dynamics of a system is relatively difficult, it is important and interesting in many cases.

The following continuous-time phytoplankton-zooplankton model suggested by the authors of \cite{Chatt}:
\begin{equation}\label{chat}
\left\{\begin{aligned}
&\frac{dP}{dt}=bP(1-\frac{P}{k})-\alpha f(P)Z,\\
&\frac{dZ}{dt}=\beta f(P)Z-rZ-\theta g(P)Z,
\end{aligned}\right.
\end{equation}
where $P$ is the density of phytoplankton and $Z$ is the density of the zooplankton population; $\alpha> 0$ and $\beta> 0$
are predation and conversion rates of the zooplankton on the phytoplankton population, respectively; $b > 0$ is the growth rate, $k > 0$ is carrying capacity of the phytoplankton; $r > 0$ is the death rate of the zooplankton; $f(P)$ represents the predator functional response; $g(P)$ represents the distribution of the toxin substances; $\theta> 0$ denotes the rate of toxin liberation by the phytoplankton population. Authors of \cite{Chatt} analyzed the local stability of the model (\ref{chat}) with different kinds of $f(u),$  $g(u)$  and proved that the release of toxin producing plankton can stop plankton blooms.

In the paper \cite{Chen}, authors investigated the model (\ref{chat}) in continuous-time by choosing $f(u)=\frac{u^h}{1+cu^h}$ (for $h = 1, 2$), $g(u) = u$ and simplifying as
$$\overline{t}=bt, \, \overline{u}=\frac{P}{k}, \, \overline{v}=\frac{\alpha k^{h-1}Z}{b}, \, \overline{c}={ck^h}, \, \overline{\beta}=\frac{\beta k^h}{b}, \, \overline{r}=\frac{r}{b}, \, \overline{\theta}=\frac{\theta k}{b}.$$

The model (\ref{chat}) takes the following form:
\begin{equation}\label{chenn}
\left\{\begin{aligned}
&\frac{du}{dt}=u(1-u)-\frac{u^hv}{1+cu^h}\\
&\frac{dv}{dt}=\frac{\beta u^hv}{1+cu^h}-rv-\theta uv.\\
\end{aligned}\right.
\end{equation}
Here  $f(u)=\frac{u}{1+cu}$  is the Holling type II predator functional response, and $f(u)=\frac{u^2}{1+cu^2}$  is the Holling type III predator functional response. In the case $\theta=0$ the global dynamics of the system (\ref{chenn})  is well studied by many mathematicians (\cite{Chen2, Hsu1, Hsu3,Ko,Peng,Wang,Zhou}). For $\theta>0$ the authors of \cite{Chen} obtained important results:
For $h=1$, it was shown that the local stability of the positive equilibrium implies global stability if there exists a unique positive equilibrium. When there exist two positive equilibria, the local stability of the positive equilibrium with small phytoplankton population density implies that the model occurs bistable
phenomenon. These results also improved  for $h=2$ under certain conditions. In \cite{SH} author considered discrete version of the model (\ref{chenn}) for $h=1$ and  proved the occurrence of Neimark-Sacker bifurcation at a positive fixed point.

In this paper, we investigate global dynamics of the system (\ref{chenn}) in discrete time.  The paper is organized as follows. In the Section 2 and Section 3, we study the global dynamics of system (\ref{chenn}) for the case $h=1$ and $h=2$ respectively. For $h=2$ we also study the types of positive fixed points and based on obtained results, we remark that the system undergoes a Neimark-Sacker bifurcation at one positive fixed point. Numerical simulations are given to illustrate the theoretical results.

Let's consider discrete-time version of the model (\ref{chenn}), which has the following form
\begin{equation}\label{h12}
V:
\begin{cases}
u^{(1)}=u(2-u)-\frac{u^hv}{1+cu^h}\\[2mm]
v^{(1)}=\frac{\beta u^hv}{1+cu^h}+(1-r)v-\theta uv.
\end{cases}
\end{equation}
In the sequel of the paper we assume that all parameters $\beta, r, \theta, c$ are positive.
It is obvious that, system (\ref{h12}) always has two nonnegative equilibria $(0; 0)$ and $(1; 0)$. From $v^{(1)}=v$ we get that $(u; v)$ is a positive fixed point  of system (\ref{h12}) if and only if  $u\in(0; 1)$ is a solution of the following equation (as in \cite{Chen}):

\begin{equation}\label{beta}
\beta=\Psi(u):=\frac{(r+\theta u)(1+c u^{h})}{u^{h}}.
\end{equation}

An easy calculation implies that (as in \cite{Chen}):

1. there exists $\overline{u}>0$ such that $\Psi'(\overline{u})=0,$ $\Psi'(u)<0$ for $u\in(0,\overline{u})$
and
$\Psi'(u)>0$ for $u\in(\overline{u},\infty);$

2. $\lim_{u\rightarrow0}\Psi(u)=\infty$ and $\Psi(1)=(1+c)(r+\theta).$

\medskip
In \cite{Chen} given the following useful lemmas.

\begin{lemma} (\cite{Chen}) Assume that $\overline{u}\geq1.$ Then system (\ref{h12}) has no positive fixed point for $\beta\leq(c+1)(r+\theta),$ and a unique positive fixed point for $\beta>(c+1)(r+\theta).$ \label{lemma1}
\end{lemma}

\begin{lemma} (\cite{Chen}) Assume that $\overline{u}<1.$ Then system (\ref{h12}) has no positive fixed point for $\beta<\Psi(\overline{u}),$  a unique positive fixed point for $\beta\geq(c+1)(r+\theta)$ or  $\beta=\Psi(\overline{u}),$  and two positive fixed points for $\beta\in(\Psi(\overline{u}), (c+1)(r+\theta)).$ \label{lemma2}
\end{lemma}

\begin{defn}\label{def1} Let $E(x,y)$ be a fixed point of the operator $F:\mathbb{R}^{2}\rightarrow\mathbb{R}^{2}$ and $\lambda_1, \lambda_2$ are eigenvalues of the Jacobian matrix $J=J_{F}$ at the point $E(x,y).$

(i) If $|\lambda_1|<1$ and $|\lambda_2|<1$ then the fixed point $E(x,y)$ is called an \textbf{attractive} or \textbf{sink};

(ii) If $|\lambda_1|>1$ and $|\lambda_2|>1$ then the fixed point $E(x,y)$ is called  \textbf{repelling} or \textbf{source};

(iii) If $|\lambda_1|<1$ and $|\lambda_2|>1$ (or $|\lambda_1|>1$ and $|\lambda_2|<1$) then the fixed point $E(x,y)$ is called  \textbf{saddle};

(iv) If either $|\lambda_1|=1$ or $|\lambda_2|=1$ then the fixed point $E(x,y)$ is called to be  \textbf{non-hyperbolic};
\end{defn}

\section{Case $h=1$}

In this case the operator (\ref{h12})  has the following form
\begin{equation}\label{h1}
V_{1}:
\begin{cases}
u^{(1)}=u(2-u)-\frac{uv}{1+cu}\\[2mm]
v^{(1)}=\frac{\beta uv}{1+cu}+(1-r)v-\theta uv.
\end{cases}
\end{equation}

\subsection{Fixed points and their stability}

\begin{pro}\label{prop1}(\cite{SH}) For the fixed points $E_0=(0,0)$ and $E_1=(1,0)$ of (\ref{h1}),  the following statements hold true:
$$E_{0}=\left\{\begin{array}{lll}
{\rm nonhyperbolic}, \ \ {\rm if} \ \  r=2\\[2mm]
{\rm saddle}, \ \ \ \ {\rm if} \ \  0<r<2\\[2mm]
{\rm repelling}, \ \ \ \  {\rm if}  \ \ r>2,
\end{array}\right.$$
$$E_{1}=\left\{\begin{array}{lll}
{\rm nonhyperbolic}, \ \ {\rm if} \ \  r+\theta=\frac{\beta}{1+c} \ \ {\rm or} \ \ r+\theta=2+\frac{\beta}{1+c}  \\[2mm]
{\rm attractive}, \ \ \ \ {\rm if} \ \  \frac{\beta}{1+c}<r+\theta<2+\frac{\beta}{1+c}\\[2mm]
{\rm saddle}, \ \ \ \  {\rm if}  \ \  {\rm otherwise}.
\end{array}\right.$$
\end{pro}

It is easy to calculate that the solution of $\Psi'(u)=0$ is $\overline{u}=\sqrt{\frac{r}{c\theta}}.$ Moreover, $\Psi(\overline{u})=(\sqrt{\theta}+\sqrt{rc})^2.$ From Lemma \ref{lemma1} and Lemma \ref{lemma2} we get the following result.

\begin{pro}\label{prop2} For the operator (\ref{h1}) the following statements hold true:

(i) If $r\geq c\theta,$ $\beta>(r+\theta)(1+c)$ or $r<c\theta,$  $\beta\geq(r+\theta)(1+c)$ or $r<c\theta,$ $\beta=(\sqrt{\theta}+\sqrt{rc})^2$ then there exists a unique positive fixed point $E_{-}=(u_{-},v_{-}),$

(ii) If $r<c\theta$ and $(\sqrt{\theta}+\sqrt{rc})^2<\beta<(c+1)(r+\theta)$ then there exist two positive fixed points $E_{-}=(u_{-},v_{-})$ and $E_{+}=(u_{+},v_{+}),$

where
$$u_{\mp}=\frac{\beta-rc-\theta\mp\sqrt{(\beta-rc-\theta)^2-4cr\theta}}{2c\theta}, \ \ v_{\mp}=(1-u_{\mp})(1+cu_{\mp}).$$
\end{pro}

We note that the characteristic equation of the Jacobian of the operator (\ref{h1}) has the form $F(\lambda, u)=\lambda^2-p(u)\lambda+q(u)=0,$ where
\begin{equation}\label{bif1}
p(u)=(1-u)\left(\frac{1+2cu}{1+cu}\right)+1, \ \ \ \ q(u)=(1-u)\left(\frac{1+2cu}{1+cu}\right)+u(1-u)\left(\frac{\beta}{(1+cu)^2}-\theta\right)
\end{equation}

\begin{lemma} (\cite{SH}) For the fixed point $E_{-}=(u_{-},v_{-})$ of the operator (\ref{h1}), the followings hold true
$$E_{-}=\left\{\begin{array}{lll}
{\rm attractive}, \ \ {\rm if} \ \  q(u_{-})<1\\[2mm]
{\rm repelling}, \ \ \ \ {\rm if} \ \  q(u_{-})>1\\[2mm]
{\rm nonhyperbolic}, \ \ \ \  {\rm if}  \ \ p(u_{-})<2, \ \ q(u_{-})=1.
\end{array}\right.$$
\end{lemma}

\begin{pro} (\cite{SH}) For the fixed point $E_{+}=(u_{+},v_{+})$ of the operator (\ref{h1}), the followings hold true
$$E_{+}=\left\{\begin{array}{lll}
{\rm saddle}, \ \ {\rm if} \ \  F(-1, u_{+})>0\\[2mm]
{\rm repelling}, \ \ \ \ {\rm if} \ \  F(-1, u_{+})<0\\[2mm]
{\rm nonhyperbolic}, \ \ \ \  {\rm if}  \ \ F(-1, u_{+})=0,
\end{array}\right.$$
\end{pro}

In \cite{SH} it was shown that the fixed point $E_{-}=(u_{-},v_{-})$ can pass through the Neimark-Sacker bifurcation.

\subsection{Dynamics of (\ref{h1})}
\begin{pro} The following sets
\begin{equation}
M_{1}=\{(u,v)\in \mathbb{R}_+^2: 0\leq u\leq2, v=0\}, \ \ M_{2}=\{(u,v)\in \mathbb{R}_+^2: u=0, v\geq0\},
\end{equation}
are invariant w.r.t. operator (\ref{h1}). In this operator, the condition $r\leq 1$ must be satisfied for invariance of $M_2.$
\end{pro}

\begin{proof} It is clear that $v^{(1)}=0$ when $v=0.$ If $v=0$ then $u^{(1)}=u(2-u)\geq0$ and the maximum value in [0,2] of the quadratic function $f(u)=u(2-u)$ is 1. Hence, the set $M_1$ is an invariant. If $u=0$ then $v^{(1)}=(1-r)v\geq0$ and $M_2$ also invariant set. The proof is completed.
\end{proof}

\subsubsection{Dynamics on invariant sets $M_1$ and $M_2$}
\hfill

\emph{Case} $M_1.$ In this case the restriction is
$$u^{(1)}=u(2-u)=f(u),  \ \ u\in[0,2].$$
Note that $f(u)$ has two fixed points $u=0$ and $u=1.$ Since $f'(u)=2-2u$ the fixed point 0 is repelling and 1 is an attracting fixed point. Moreover, the equation
$$\frac{f^2(u)-u}{f(u)-u}=0 \Rightarrow \frac{u^4-4u^3+6u^2-3u}{u^2-u}=0 \Rightarrow u^2-3u+3=0 $$
has no real solution. Consequently, the function $f$ does not have two periodic points and, by Sarkovskii's theorem, there is no any periodic point (except fixed point).

Let's study the dynamics of $f(u):$ If $u\in(0,1)$ then $0<u<f(u)<1.$ So
\[
0<f(u)<f^2(u)<1 \Rightarrow  \cdots \Rightarrow 0<f^n(u)<f^{n+1}(u)<1.
\]
Thus, the sequence $f^n(u)$ is monotone increasing and bounded from above, which has the limit. Since 1 is a unique fixed point in $(0,1],$ we have that
\[
\lim_{n\to\infty}f^{n}(u)=1.
\]
If $u$ lies in the interval (1,2) then $f(u)$ lies in (0,1), so that the previous argument implies $f^n(u)=f^{n-1}(f(u))\rightarrow1$  as $n\rightarrow\infty.$
Hence, 1 is globally attractive fixed point. (similar to quadratic family, \cite{De}, page 32 ).

\emph{Case} $M_2.$ In this case the restriction is $$v^{(1)}=(1-r)v,  \ \ 0<r\leq1$$
and $v^{(n)}=(1-r)^nv\rightarrow0$ as $n\rightarrow\infty.$ Thus, the dynamics on $M_1$ and $M_2$ is clear.

From $v^{(1)}\geq0$ we obtain  the inequality

\begin{equation}\label{fpeq}
c\theta u^2-(\beta+c-\theta-rc)u+r-1\leq0.
\end{equation}
for which we have the solution that $u\in[\widehat{u}_{-},\widehat{u}_{+}],$ where

\[
\widehat{u}_{\mp}=\frac{\beta+c-\theta-rc\mp\sqrt{(\beta+c-\theta-rc)^2+4c\theta(1-r)}}{2c\theta}.
\]

Let's classify parameters set as following:

\textbf{Class A:}  $\widehat{u}_{-}<0,\ \ \widehat{u}_{+}>1:$
\medskip

(a1)  $0<\theta\leq1, \ \ 0<r\leq1-\theta, \ \ \beta>0, \ \  c>0;$

(a2)  $0<\theta\leq1, \ \ 1-\theta<r<1, \ \ \beta>r+\theta-1, \ \  0<c<\frac{\beta}{r+\theta-1}-1;$

(a3)  $\theta>1, \ \ 0<r<1, \ \ \beta>r+\theta-1, \ \  0<c<\frac{\beta}{r+\theta-1}-1.$

\medskip
\textbf{Class B:} $0<\widehat{u}_{-}<1,\ \ \widehat{u}_{+}>1:$
\medskip

(b1)  $\theta>0, \ \ r>1, \ \ \beta>r+\theta-1, \ \  0<c<\frac{\beta}{r+\theta-1}-1;$

\medskip
\textbf{Class C:}  $\widehat{u}_{-}\leq0,\ \ \widehat{u}_{+}\geq2:$
\medskip

(c1)  $0<\theta\leq\frac{1}{2}, \ \ 0<r\leq1-2\theta, \ \ \beta>0, \ \  c>0;$

(c2)  $0<\theta\leq\frac{1}{2}, \ \ 1-2\theta < r\leq1, \ \ \beta>\frac{r+2\theta-1}{2}, \ \ 0<c\leq\frac{\beta}{r+2\theta-1}-\frac{1}{2} ;$

(c3)  $\theta>\frac{1}{2}, \ \ 0< r\leq1, \ \ \beta>\frac{r+2\theta-1}{2}, \ \ 0<c\leq\frac{\beta}{r+2\theta-1}-\frac{1}{2} ;$

\medskip
\textbf{Class D:}  $\widehat{u}_{-}<0,\ \ 1<\widehat{u}_{+}<2:$
\medskip

(d1)  $0<\theta\leq\frac{1}{2}, \ \ 1-2\theta < r\leq1-\theta, \ \ 0<\beta\leq\frac{r+2\theta-1}{2}, \ \  c>0;$

(d2)  $0<\theta\leq\frac{1}{2}, \ \ 1-2\theta < r\leq1-\theta, \ \ \beta>\frac{r+2\theta-1}{2}, \ \ c>\frac{\beta}{r+2\theta-1}-\frac{1}{2} ;$

(d3)  $\frac{1}{2}<\theta\leq1, \ \ 0 < r\leq1-\theta, \ \ 0<\beta\leq\frac{r+2\theta-1}{2}, \ \  c>0;$

(d4)  $\frac{1}{2}<\theta\leq1, \ \ 0 < r\leq1-\theta, \ \ \beta>\frac{r+2\theta-1}{2}, \ \ c>\frac{\beta}{r+2\theta-1}-\frac{1}{2} ;$

(d5)  $0<\theta\leq1,  \ \ 1-\theta<r<1, \ \ r+\theta-1<\beta\leq\frac{r+2\theta-1}{2}, \ \  0<c<\frac{\beta}{r+\theta-1}-1;$

(d6)  $0<\theta\leq1,   \ \ 1-\theta<r<1, \ \ \beta>\frac{r+2\theta-1}{2}, \ \  \frac{\beta}{r+2\theta-1}-\frac{1}{2}<c<\frac{\beta}{r+\theta-1}-1;$

(d7)  $\theta>1,  \ \ 0<r<1, \ \ r+\theta-1<\beta\leq\frac{r+2\theta-1}{2}, \ \  0<c<\frac{\beta}{r+\theta-1}-1;$

(d8)  $\theta>1,   \ \ 0<r<1, \ \ \beta>\frac{r+2\theta-1}{2}, \ \  \frac{\beta}{r+2\theta-1}-\frac{1}{2}<c<\frac{\beta}{r+\theta-1}-1.$

\medskip

Note that for other parameter values we have $v^{(1)}<0.$

\begin{pro}\label{invm} Let $r\geq c\theta,$  $\beta\leq(c+1)(r+\theta)$ or $r\leq c\theta,$  $\beta\leq(\sqrt{\theta}+\sqrt{cr})^2$. If one of the (a1)--(a3) is satisfied then the following sets are invariant w.r.t. operator (\ref{h1}):

(i)
\[
M_3=\left\{(u,v)\in \mathbb{R}^2: 0\leq u\leq1, \, 0\leq v\leq2\right\},
\]

(ii) if $c\leq1/2$ then
\[
M_4=\left\{(u,v)\in \mathbb{R}^2: 0\leq u\leq1, \, 0\leq v\leq(2-u)(1+cu)\right\}.
\]
\end{pro}

\begin{proof} (i). First, we will show that for any  $(u,v)\in M_3$,  it follows that $(u^{(1)},v^{(1)})\in M_3.$ Let $(u,v)\in M_3.$ From $v\leq2\leq(2-u)(1+cu)$ for $u\in[0,1],$ we have that $u^{(1)}\geq0.$ Moreover, the maximum value of the function $x(2-x)$ is 1, so
$u^{(1)}=u(2-u)-\frac{uv}{1+cu}\leq u(2-u)\leq1.$ Thus, we have $0\leq u^{(1)}\leq1.$

If one of the conditions $(a1)-(a3)$ is satisfied then $\widehat{u}_{-}<0$ and $\widehat{u}_{+}>1.$ Consequently, $v^{(1)}\geq0$  for all $u\in[0,1].$  Moreover, from $r\geq c\theta$ and $\beta\leq(c+1)(r+\theta)$ (or $r\leq c\theta,$  $\beta\leq(\sqrt{\theta}+\sqrt{cr})^2$) we obtain that $\beta\leq\Psi(u)$ for all $u\in[0,1],$ i.e., $v^{(1)}\leq v\leq2$ for all $u\in[0,1].$ Thus $(u^{(1)},v^{(1)})\in M_3.$

(ii) Let $c\leq1/2$ and $(u,v)\in M_4.$ Obtaining the conditions $0\leq u^{(1)}\leq1$ and $v^{(1)}\geq0$ are similar to the previous case. From $c\leq1/2$ we have that the function $(2-x)(1+cx)$ is decreasing in the interval $[0,1].$ In addition, if $v\leq(1-u)(1+cu)$ then $v\leq(2-u)(1+cu)$ is obvious, if $v>(1-u)(1+cu)$ then $u^{(1)}<u$ and since the function $(2-x)(1+cx)$ is decreasing we have that $(2-u^{(1)})(1+cu^{(1)})\geq(2-u)(1+cu).$ Finally, we get
\[
v^{(1)}\leq v \leq (2-u)(1+cu)\leq(2-u^{(1)})(1+cu^{(1)}),
\]
which supplies that $(u^{(1)},v^{(1)})\in M_4.$ The proof is complete.
\end{proof}

\begin{rk} If $c>1/2$ then the set $M_4$ is not invariant w.r.t. operator (\ref{h1}). For example, the point $(u,v)=(0.15; 6)\in M_4$ while $(u^{(1)},v^{(1)})\approx(0.05; 4.66)\notin M_4.$
\end{rk}

\begin{pro}\label{prop6} Let $r\geq c\theta,$  $\beta\leq(c+1)(r+\theta)$ or $r\leq c\theta,$  $\beta\leq(\sqrt{\theta}+\sqrt{cr})^2.$ If

(i) one of the conditions $(a1)-(a3)$ is satisfied and $(u^{(0)},v^{(0)})\in M_3$;

(ii) one of the conditions $(a1)-(a3)$ is satisfied, $c\leq1/2$ and $(u^{(0)},v^{(0)})\in M_4$;

(iii) the condition $(b1)$ is satisfied and $\widehat{u}_{-}\leq u^{(0)}\leq1,$ $0\leq v^{(0)}\leq(1-u^{(0)})(1+cu^{(0)})$;\\
then for the initial point  $(u^{(0)},v^{(0)}),$ the trajectory has the following limit
\[
\lim_{n\to\infty}V_{1}^{n}(u^{(0)},v^{(0)})=E_1=\left(1,0\right).
\]
\end{pro}

\begin{proof} Let  $r\geq c\theta,$  $\beta\leq(c+1)(r+\theta)$ (or $r\leq c\theta,$  $\beta\leq(\sqrt{\theta}+\sqrt{cr})^2$). Then  by Lemma \ref{lemma1} and  Lemma \ref{lemma2} it follows that there is no positive fixed point. Moreover, $\beta<\Psi(u)$ for all $u\in[0,1]$ which supplies that the sequence  $v^{(n)}$ is decreasing, so it has limit. Since $(\sqrt{\theta}+\sqrt{cr})^2\leq(c+1)(r+\theta)$ we have that $\beta\leq(c+1)(r+\theta)$ is always true for the conditions of proposition. From this we get that $E_1=\left(1,0\right)$ is an attracting fixed point (according to Proposition \ref{prop1}). Then there exists a neighbourhood $U$ of $E_1$ such that for any initial point from $U,$ the trajectory converges to $E_1.$

Now we divide the invariant set $M_4$ into three parts $S_1, S_2$ and $S_3,$ (as in Fig. \ref{mset}), where $M_3=S_1\cup S_2,$  $M_4=S_1\cup S_2\cup S_3=M_3\cup S_3,$ and
\[
S_1=\left\{(u,v)\in \mathbb{R}^2: 0< u\leq1, \, 0\leq v\leq(1-u)(1+cu)\right\},
\]
\[
S_2=\left\{(u,v)\in \mathbb{R}^2: 0< u\leq1, \, (1-u)(1+cu)<v\leq2\right\},
 \]
 \[
S_3=\left\{(u,v)\in \mathbb{R}^2: 0< u\leq1, \, 2<v\leq(2-u)(1+cu)\right\}.
 \]
If $(u^{(0)},v^{(0)})\in S_1$ (resp. $(u^{(0)},v^{(0)})\in S_2\cup S_3$) then it is obvious that  $u^{(1)}\geq u^{(0)}$ (resp. $u^{(1)}\leq u^{(0)}$). Thus, the sequence $u^{(n)}$ is increasing in $S_1$ and decreasing in $S_2\cup S_3.$ Recall that from the conditions of proposition, the sequence $v^{(n)}$ is always decreasing, the sets $M_3$ and $M_4$ are invariant sets and there is no any positive fixed point. To sum up, we have that for any initial point $(u^{(0)},v^{(0)})\in M_3$ (or  $(u^{(0)},v^{(0)})\in M_4$ with condition $c\leq1/2$),  the trajectory falls into $U$ after some finite steps and then converges to the fixed point $(1,0).$ If the condition $(b1)$ is satisfied and $\widehat{u}_{-}\leq u^{(0)}\leq1,$ $0\leq v^{(0)}\leq(1-u^{(0)})(1+cu^{(0)})$ then $u^{(n)}\geq\widehat{u}_{-},$ $v^{(n)}\geq0$ for all $n\in\mathbb{ N}$ and the trajectory falls into $U$ after some finite steps. The proof is complete.
\end{proof}

\begin{figure}[h!]
  \centering
  \includegraphics[width=7cm]{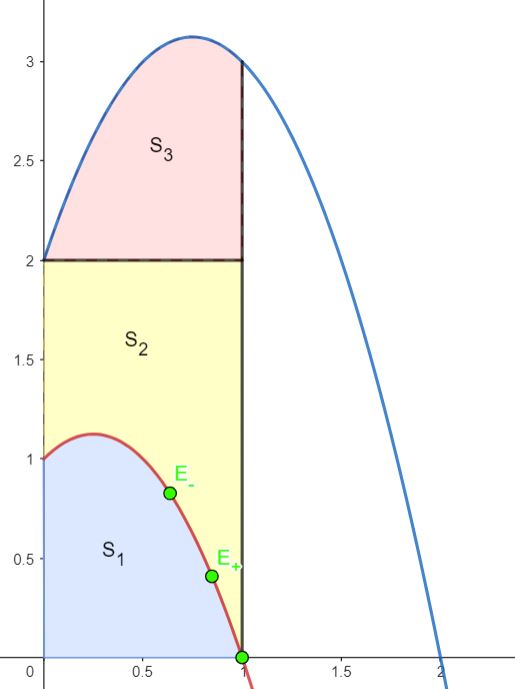}\\
  \caption{The invariant set $M_4$ is divided into three parts $S_1, S_2$ and $S_3.$}\label{mset}
\end{figure}

\begin{pro} \label {prop8}Let $r\geq 4c\theta,$  $\beta\leq\frac{(2c+1)(r+2\theta)}{2}$ or $r\leq 4c\theta,$  $\beta\leq(\sqrt{\theta}+\sqrt{cr})^2$ and $(u^{(0)}, v^{(0)})$ be an initial point. If

(i) one of the (c1)--(c3) is satisfied and $1<u^{(0)}<2,$ $0<v^{(0)}\leq(2-u^{(0)})(1+cu^{(0)});$

(ii) one of the (d1)--(d8) is satisfied and $1<u^{(0)}\leq\widehat{u}_{+},$ $0<v^{(0)}\leq(2-u^{(0)})(1+cu^{(0)});$\\
 then the trajectory converges to the fixed point (1,0) when $(u^{(1)}, v^{(1)})\in M_3$  or $(u^{(1)}, v^{(1)})\in M_4\setminus M_3$ with $c\leq1/2$.
\end{pro}

\begin{proof} Let $r\geq 4c\theta,$  $\beta\leq\frac{(2c+1)(r+2\theta)}{2}$ or $r\leq 4c\theta,$  $\beta\leq(\sqrt{\theta}+\sqrt{cr})^2$. These conditions are taken from $\overline{u}=\sqrt{\frac{r}{c\theta}}\geq2,  \ \ \beta<\Psi(2)$ or $\overline{u}\leq2,  \ \ \beta<\Psi(\overline{u})$ respectively. From this we get that  $\beta<\Psi(u)$ for all $u\in[0,2],$  and   $v^{(1)})\leq v.$ In addition, if parameters satisfy one of the conditions (c1)--(c3) then $v^{(1)})\geq0.$  The condition $0\leq u^{(1)}\leq1$ is obvious. By easy calculation it can be shown that from the conditions $r\geq 4c\theta,$  $\beta\leq\frac{(2c+1)(r+2\theta)}{2}$ (resp. $r\leq 4c\theta,$  $\beta\leq(\sqrt{\theta}+\sqrt{cr})^2$) we get $r\geq c\theta,$  $\beta\leq(c+1)(r+\theta)$ (resp. $r\leq c\theta,$  $\beta\leq(\sqrt{\theta}+\sqrt{cr})^2$). So, the convergence of trajectory to  (1,0) is straightforward from the previous proposition.   The proof is complete.
\end{proof}

\begin{pro} \label{prop8} Let $\beta>(c+1)(r+\theta)$ (or $r<c\theta$ and $\beta\in\{(c+1)(r+\theta),(\sqrt{\theta}+\sqrt{cr})^2\}$) and $q(u_{-})<1$ and $(u^{(0)},v^{(0)})$ be an initial point. If

(i) one of the conditions (a1)--(a3) is satisfied and
\[
(u^{(0)},v^{(0)})\in\left\{(u,v)\in \mathbb{R}^2: 0< u< 1, \, 0<v\leq(1-u)(1+cu)\right\};
\]

(ii) one of the conditions (a1)--(a3) is satisfied and
\[
(u^{(0)},v^{(0)})\in\left\{(u,v)\in \mathbb{R}^2: 0< u< u_{-}, \, (1-u)(1+cu)\leq v\leq2\right\};
\]

(iii) one of the conditions (a1)--(a3) is satisfied, $c\leq1/2$ and
\[
(u^{(0)},v^{(0)})\in\left\{(u,v)\in \mathbb{R}^2: 0< u< u_{-}, \, (1-u)(1+cu)\leq v\leq(2-u)(1+cu)\right\};
\]

(iv) the condition (b1) is satisfied and
\[
(u^{(0)},v^{(0)})\in\left\{(u,v)\in \mathbb{R}^2: \widehat{u}_{-}\leq u< u_{-}, \, 0<v\leq(1-u)(1+cu)\right\};
\]
then the trajectory has the following limit
\[
\lim_{n\to\infty}V_{1}^{n}(u^{(0)},v^{(0)})=E_{-}=(u_{-},v_{-}).
\]
\end{pro}

\begin{proof} If $\beta>(c+1)(r+\theta)$ (or $r<c\theta$ and $\beta\in\{(c+1)(r+\theta),(\sqrt{\theta}+\sqrt{cr})^2\}$) then according to Proposition \ref{prop2}, there exists a unique positive fixed point $E_{-}=(u_{-},v_{-}).$ Moreover, from $q(u_{-})<1$ it follows that $E_{-}$ is an attracting point. Thus, there exists a neighbourhood $U(E_{-})$ of $E_{-}$, such that for any initial point taken from $U(E_{-}),$ the trajectory converges to $E_{-}.$   With the conditions to parameters in (i)-(iv) the trajectory remains in $M_4.$
 Note that $u_{-}$ is a solution of $c\theta u^2-(\beta-rc-\theta)u+r=0$ (i.e., $\beta=\Psi(u)$), and $c\theta u^2-(\beta-rc-\theta)u+r>0$ for all $u\in(0,u_{-})$ which is equivalent to $\beta<\Psi(u).$ Similarly, $\beta>\Psi(u)$ for all $u\in(u_{-},1).$ Thus, the sequence $v^{(n)}$ is decreasing in
 $\left\{(u,v)\in \mathbb{R}^2: 0< u< u_{-}, \, 0<v\leq(2-u)(1+cu)\right\}$ and increasing in  $\left\{(u,v)\in \mathbb{R}^2: u_{-}< u< 1, \, 0<v\leq(2-u)(1+cu)\right\}.$ Note also the sequence $u^{(n)}$ is increasing in the set $\left\{(u,v)\in \mathbb{R}^2: 0< u< 1, \, 0<v\leq(1-u)(1+cu)\right\}$ and decreasing in $\left\{(u,v)\in \mathbb{R}^2: 0< u< 1, \, (1-u)(1+cu)<v\leq(2-u)(1+cu)\right\}.$

(i). Let one of the (a1)--(a3) is satisfied and
  \[
 (u^{(0)},v^{(0)})\in\left\{(u,v)\in \mathbb{R}^2: 0< u< 1, \, 0<v\leq(1-u)(1+cu)\right\}\setminus U(E_{-}).
 \]
  From above discussion we conclude that the trajectory falls into $U(E_{-})$ after some finite steps and then goes to the fixed point $E_{-}=(u_{-},v_{-}).$  (as in Fig.\ref{fig1}). In cases (ii) and (iii), both sequences $u^{(n)}, v^{(n)}$ decrease to those as long as the trajectory falls into  $\left\{(u,v)\in \mathbb{R}^2: 0< u< 1, \, 0<v\leq(1-u)(1+cu)\right\}$ and then the situation is the same. The last case is very similar to the first case.
 The proof is complete.
\end{proof}

\begin{figure}
  \centering
  \includegraphics[width=10cm]{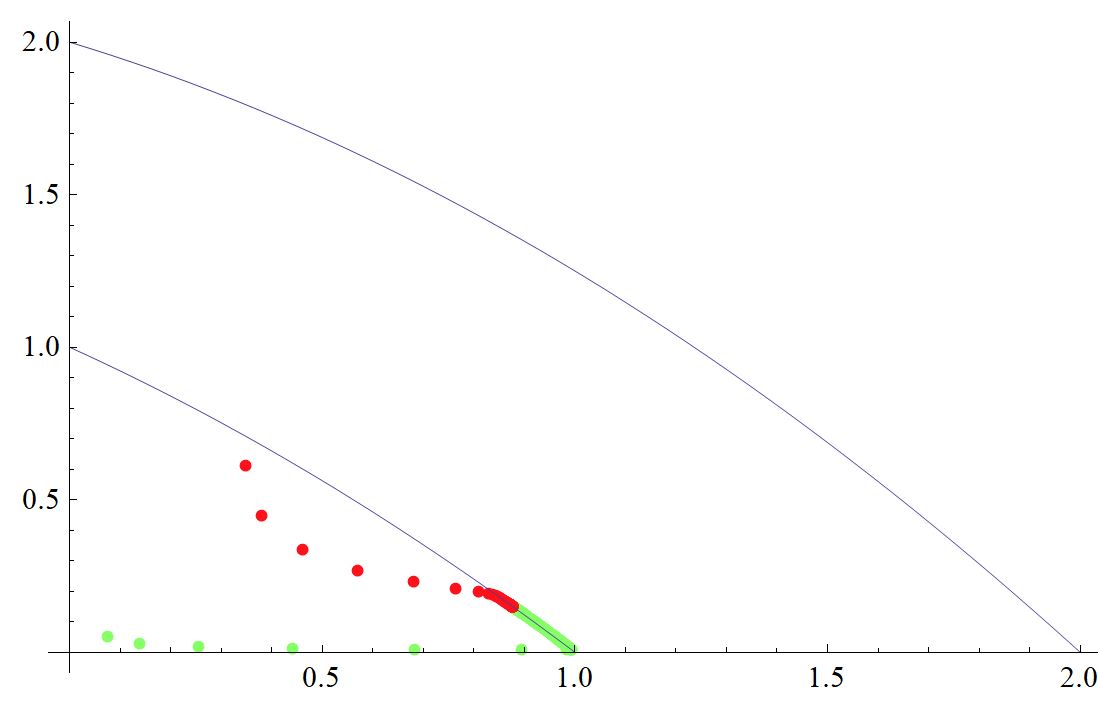}\\
  \caption{$\beta=1,  c=0.25, r=0.5, \theta=0.25, (u_{-},v_{-})\approx(0.876894, 0.150093)$. Green is the trajectory of initial point $u^{(0)}=0.04, v^{(0)}=0.1,$ while red is the trajectory of $u^{(0)}=0.4, v^{(0)}=0.8.$  }\label{fig1}
\end{figure}

\begin{pro} Let $r< c\theta,$ $(\sqrt{\theta}+\sqrt{rc})^2<\beta<(c+1)(r+\theta)$ and $q(u_{-})<1$ and $(u^{(0)},v^{(0)})$ be an initial point. If

(i) one of the conditions (a1)--(a3) is satisfied and
\[
(u^{(0)},v^{(0)})\in\left\{(u,v)\in \mathbb{R}^2: 0< u< 1, \, 0<v\leq(1-u)(1+cu)\right\};
\]

(ii) one of the conditions (a1)--(a3) is satisfied and
\[
(u^{(0)},v^{(0)})\in\left\{(u,v)\in \mathbb{R}^2: 0< u< u_{-}, \, (1-u)(1+cu)\leq v\leq2\right\};
\]

(iii) one of the conditions (a1)--(a3) is satisfied, $c\leq1/2$ and
\[
(u^{(0)},v^{(0)})\in\left\{(u,v)\in \mathbb{R}^2: 0< u< u_{-}, \, (1-u)(1+cu)\leq v\leq(2-u)(1+cu)\right\};
\]

(iv) the condition (b1) is satisfied and
\[
(u^{(0)},v^{(0)})\in\left\{(u,v)\in \mathbb{R}^2: \widehat{u}_{-}\leq u< u_{-}, \, 0<v\leq(1-u)(1+cu)\right\};
\]
Then the stable manifold of $(u_{+},v_{+})$ divides the set $M_4$ into two parts  $\Omega_1$ (which includes $E_{-}$) and $\Omega_2$ such that the trajectory converges to the fixed point $E_{-}=(u_{-},v_{-})$ if $(u^{(0)},v^{(0)})\in\Omega_1$ and the trajectory converges to the fixed point $(1,0)$ if $(u^{(0)},v^{(0)})\in\Omega_2.$
\end{pro}

\begin{proof} If $r<c\theta$ and $(\sqrt{\theta}+\sqrt{rc})^2<\beta<(c+1)(r+\theta)$ then according to the Lemma \ref{lemma1} and  Lemma \ref{lemma2} there exist two positive fixed points $E_{-}=(u_{-},v_{-})$, $E_{+}=(u_{+},v_{+}),$ and from $q(u_{-})<1,$  $\beta<(c+1)(r+\theta)$ we get that the fixed points $E_{-}=(u_{-},v_{-})$ and $(1,0)$ are attracting points. According to Theorem 2.4 in \cite{Chen}, if a fixed point $(u_{-},v_{-})$ is attractive, then the fixed point $(u_{+},v_{+})$ is a saddle fixed point. Then for the saddle point $(u_{+},v_{+})$ the parabola $(1-u)(1+cu)$ plays the role of an unstable manifold and there is another stable curve $\gamma$ which divides the set $M_4$ into two parts (as in \cite{Chen}). Finding this stable curve requires to find the solution of ODE $\dot{v}/\dot{u}$ passing through the point $E_{+}$ (in Fig.\ref{fig2} we tried to find it using a numerical solution.) Here also as in Proposition \ref{prop8}, first the trajectory approaches to the parabola $(1-x)(1+cx),$ and then it goes to one of the fixed points $E_{-}$ or $E_1.$  It is obvious that in a region containing a fixed point $E_{-}$, the trajectory converges to it, and in another region the trajectory converges to $(1,0)$.

\end{proof}

\begin{figure}[h!]
  \centering
  \includegraphics[width=7cm]{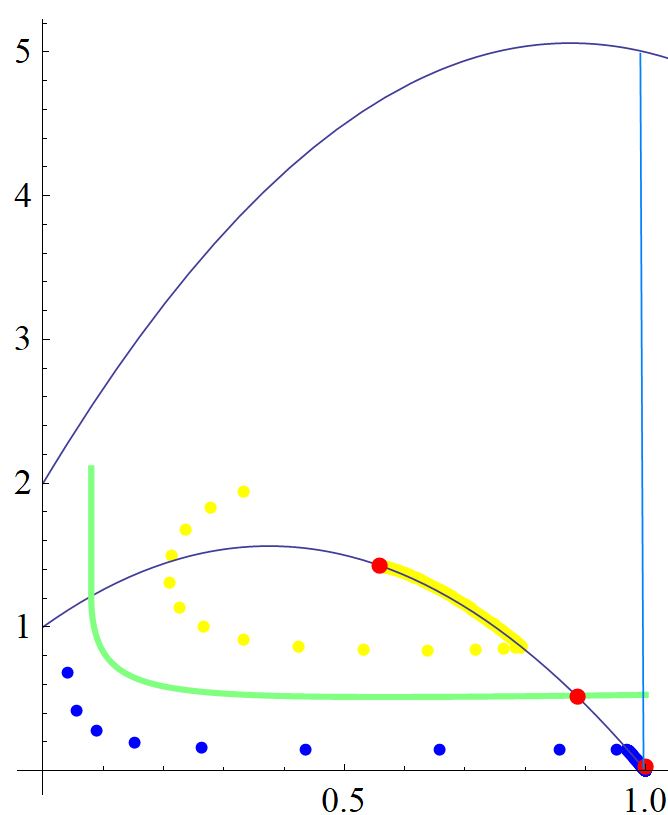}\\
  \caption{$\beta=3.7,  c=4, r=0.5, \theta=0.25, (u_{-},v_{-})\approx(0.564922, 1.41822), (u_{+},v_{+})\approx(0.885078, 0.521718)$. Yellow is the trajectory of initial point $u^{(0)}=0.4, v^{(0)}=2,$ while blue is the trajectory of $u^{(0)}=0.04, v^{(0)}=1.1.$ Green line is a stable manifold passing through $(u_{+},v_{+})$. }\label{fig2}
\end{figure}

\begin{rk} If $q(u_{-})>1$ then the fixed point $(u_{-},v_{-})$ is a repelling point. In this situation can occur the Neimark-Sacker bifurcation and appears an invariant closed curve about  $(u_{-},v_{-})$  when one of the parameters $\beta, r, \theta, c$ varies in a small neighbourhood of the value of that parameter which satisfies the equation $q(u_{-})=1.$ (In \cite{SH} it was shown w.r.t. parameter $\theta$).
\end{rk}

\begin{rk} For other cases of parameter values that were not considered, the convergence of the trajectory in $M_3$ depends not only on the parameters, but also on the initial point.
\end{rk}

\section{Case $h=2$}

In this case the operator (\ref{h12})  has the following form
\begin{equation}\label{h2}
V_{2}:
\begin{cases}
u^{(1)}=u(2-u)-\frac{u^2v}{1+cu^2}\\[2mm]
v^{(1)}=\frac{\beta u^2v}{1+cu^2}+(1-r)v-\theta uv.
\end{cases}
\end{equation}

\subsection{Fixed points and their stability} By simple calculations it can be shown that for the operator (\ref{h2}), Proposition \ref{prop1} is true for fixed points $(0,0)$ and $(1,0).$ Moreover, the sets $\{(u,v)\in \mathbb{R}_+^2: 0\leq u\leq2, v=0\}$ and $\{(u,v)\in \mathbb{R}_+^2: u=0, v\geq0\}$ are also invariant w.r.t. operator $V_2,$ and the dynamics in these sets is the same as in the case $h=1.$ By Lemmas \ref{lemma1} and \ref{lemma2}, the operator $V_2$ can have one ($E_{-}=(u_{-},v_{-})$) or two   ($E_{-}=(u_{-},v_{-})$ and $E_{+}=(u_{+},v_{+})$) positive fixed points. Let's study the type of these fixed points.
From the system (\ref{h2}) it is easy to see that $u_{\mp}$ are solutions between 0 and 1 of the following cubic equation:

\begin{equation}\label{cubic}
\theta c u^3+(rc-\beta)u^2+\theta u+r=0
\end{equation}
and
\begin{equation}\label{v}
v_{\mp}=\frac{(1-u_{\mp})(1+cu_{\mp}^2)}{u_{\mp}}.
\end{equation}

Simple calculation derives that the Jacobian of the operator (\ref{h2}) is

\begin{equation}\label{jac}
J(u_{\mp},v_{\mp})=\begin{bmatrix}
\frac{2cu_{\mp}^2(1-u_{\mp})}{1+cu_{\mp}^2} & -\frac{u_{\mp}^2}{1+cu_{\mp}^2}\\
\left(\frac{2\beta u_{\mp}}{(1+cu_{\mp}^2)^2}-\theta\right)v_{\mp} & 1
\end{bmatrix}
\end{equation}
and characteristic polynomial is

\begin{equation}\label{charpol}
\overline{F}(\lambda,u_{\mp})=\lambda^2-\left(1+\frac{2cu_{\mp}^2(1-u_{\mp})}{1+cu_{\mp}^2}\right)\lambda+\frac{2cu_{\mp}^2(1-u_{\mp})}{1+cu_{\mp}^2}+u_{\mp}(1-u_{\mp})\left(\frac{2\beta u_{\mp}}{(1+cu_{\mp}^2)^2}-\theta\right).
\end{equation}
We denote $\overline{F}(\lambda,u_{\mp})=\lambda^2-\overline{p}(u_{\mp})\lambda+\overline{q}(u_{\mp}),$ where
\begin{equation}\label{pq}
\overline{p}(u_{\mp})=1+\frac{2cu_{\mp}^2(1-u_{\mp})}{1+cu_{\mp}^2}, \quad
\overline{q}(u_{\mp})= \frac{2cu_{\mp}^2(1-u_{\mp})}{1+cu_{\mp}^2}+u_{\mp}(1-u_{\mp})\left(\frac{2\beta u_{\mp}}{(1+cu_{\mp}^2)^2}-\theta\right).
\end{equation}

We give the following useful lemma.
\begin{lemma}[Lemma 2.1, \cite{Cheng}]\label{keylem} Let $F(\lambda)=\lambda^2+B\lambda+C,$ where $B$ and $C$ are two real constants. Suppose $\lambda_1$ and $\lambda_2$ are two roots of $F(\lambda)=0.$ Then the following statements hold.
\begin{itemize}
 \item[(i)]  If $F(1)>0$ then
{\begin{itemize}
\item[(i.1)]
$|\lambda_1|<1$ and $|\lambda_2|<1$ if and only if $F(-1)>0$ and $C<1;$
\item[(i.2)]
 $\lambda_1=-1$ and $\lambda_2\neq-1$ if and only if $F(-1)=0$ and $B\neq2;$
\item[(i.3)]
$|\lambda_1|<1$ and $|\lambda_2|>1$ if and only if $F(-1)<0;$
\item[(i.4)]
$|\lambda_1|>1$ and $|\lambda_2|>1$ if and only if $F(-1)>0$ and $C>1;$
\item[(i.5)]
$\lambda_1$ and $\lambda_2$ are a pair of conjugate complex roots and $|\lambda_1|\!=\!|\lambda_2|\!=\!1$ if and only
       if $-2<B<2$ and $C=1;$
\item[(i.6)]
 $\lambda_1=\lambda_2=-1$ if and only if $F(-1)=0$ and $B=2.$
\end{itemize}}
\item[(ii)]  If $F(1)=0,$ namely, 1 is one root of $F(\lambda)=0,$ then the other root $\lambda$ satisfies
$|\lambda|=(<,>)1$ if and only if $|C|=(<,>)1.$
\item[(iii)]  If $F(1)<0,$ then $F(\lambda)=0$ has one root lying in $(1;\infty).$ Moreover,
\begin{itemize}
\item[(iii.1)]
 the other root $\lambda$ satisfies $\lambda<(=)-1$ if and only if $F(-1)<(=)0;$
\item[(iii.2)]
 the other root $\lambda$ satisfies $-1<\lambda<1$ if and only if $F(-1)>0.$
 \end{itemize}
\end{itemize}
\end{lemma}

\begin{pro}\label{type} For the fixed points $E_{\mp}=(u_{\mp},v_{\mp})$ of the operator (\ref{h2}), the followings hold true

(i)
$$E_{-}=\left\{\begin{array}{lll}
&{\rm attractive}, ~~& {\rm if} \ \  \overline{q}(u_{-})<1\\
&{\rm repelling}, ~~& {\rm if} \ \  \overline{q}(u_{-})>1\\
&{\rm nonhyperbolic}, ~~&  {\rm if}   \ \ \overline{q}(u_{-})=1,
\end{array}\right.$$

(ii) The fixed point $E_{+}=(u_{+},v_{+})$ is a saddle point if it exists.
\end{pro}

\begin{proof} First, we will show $\overline{F}(1,u_{-})>0.$ From (\ref{charpol}) we get that $\overline{F}(1,u_{-})>0$  is equivalent to $\beta>\frac{\theta(1+cu_{-}^2)^2}{2u_{-}}.$ In addition, for the existence of a positive fixed point by Lemmas \ref{lemma1} and \ref{lemma2} it is necessary that $\beta>\Psi(\overline{u}),$ where $\overline{u}$ is a positive solution (it is unique) of the equation (as in \cite{Chen})
\[
\theta cx^3-\theta x-2r=0.
\]
Let's denote by $\phi(x)=\frac{\theta(1+cx^2)^2}{2x}$ and consider $\Psi(x)-\phi(x):$

\begin{equation}
\Psi(x)-\phi(x)=\frac{(r+\theta x)(1+c x^{2})}{x^{2}}-\frac{\theta(1+cx^2)^2}{2x}=\frac{(1+cx^2)(-\theta cx^3+\theta x+2r)}{2x^2}.
\end{equation}
From this we obtain that the functions $\Psi(x), \ \ \phi(x)$ intersect in the first quadrant only at $\overline{u}.$  Moreover, $\theta cx^3-\theta x-2r<0$ for $x\in(0,\overline{u}),$ ($\Psi(x)$ \emph{is a decreasing}) and $\theta cx^3-\theta x-2r>0$ for $x>\overline{u}.$ ($\Psi(x)$ \emph{is an increasing}). It means that $\Psi(x)>\phi(x)$ for any $x\in(0,\overline{u}).$ Recall that $u_{\mp}$ is a solution of $\beta=\Psi(u),$  so $u_{-}<\overline{u}$ and $u_{+}>\overline{u}.$ (as in Fig.\ref{fig3}) Thus, we proved that $\beta>\frac{\theta(1+cu_{-}^2)^2}{2u_{-}}$ and $\beta<\frac{\theta(1+cu_{+}^2)^2}{2u_{+}}$ which means $\overline{F}(1,u_{-})>0$ and $\overline{F}(1,u_{+})<0.$  If $\overline{F}(1,u_{-})>0$  then $\overline{F}(-1,u_{-})>0$ is straightforward. In addition,  from $\overline{F}(1,u_{-})>0$ and  $\overline{F}(-1,u_{-})>0$ we obtain that $0<\overline{p}(u_{-})<1+\overline{q}(u_{-})$ which is sufficient that $-2<\overline{p}(u_{-})<2$ when $\overline{q}(u_{-})=1.$ Thus, according to Lemma \ref{keylem}, we obtain a proof of statement (i) of the proposition. Above we showed that $\overline{F}(1,u_{+})<0.$ For our further results we consider only the case $0<\theta<1,$ from this we obtain that $\overline{F}(-1,u_{+})>0.$ According to Lemma \ref{keylem}, we get that $E_{+}=(u_{+},v_{+})$ is a saddle point. The proof is complete.
\end{proof}
\begin{figure}[h!]
  \centering
  \includegraphics[width=6cm]{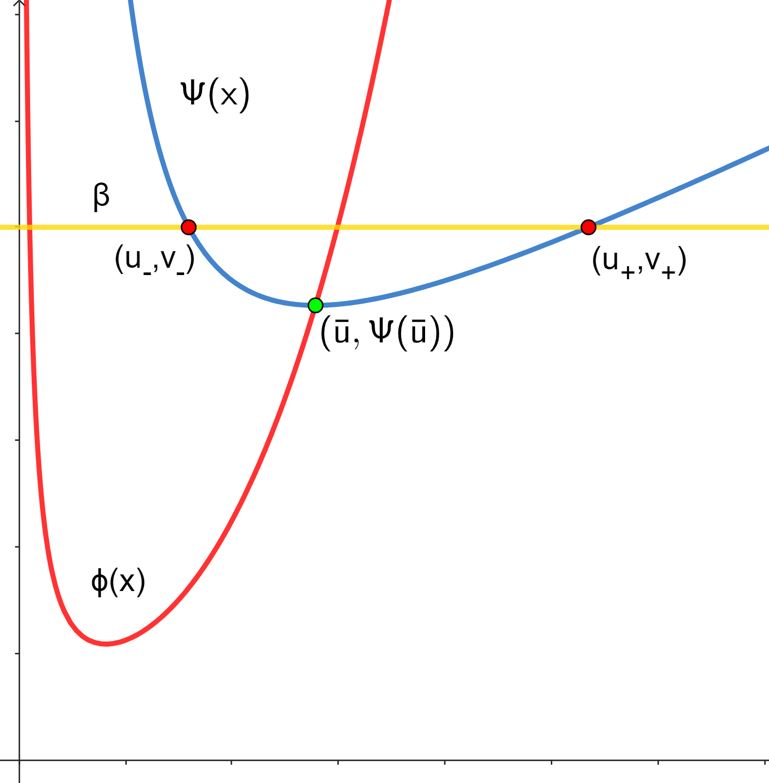}\\
  \caption{Blue line -- $\Psi(x)$, red line -- $\phi(x)$}\label{fig3}
\end{figure}

\subsection{Dynamics of (\ref{h2})}
\begin{pro}\label{invn} Let $r+\theta\leq1$ and $c\leq27/4$. If $\overline{u}\geq1,$  $0<\beta\leq(c+1)(r+\theta)$ or $\overline{u}<1,$  $0<\beta\leq\Psi(\overline{u})$ then the following set is an invariant w.r.t operator (\ref{h2}):

\[
N=\left\{(u,v)\in R^2: 0\leq u\leq1, \, 0\leq v\leq\frac{(2-u)(1+cu^2)}{u}\right\}.
\]
Moreover, the trajectory converges to the fixed point $(1,0)$ for any initial point $(u^{(0)}, v^{(0)})\in N.$
\end{pro}
\begin{proof} Let $(u^{(0)}, v^{(0)})\in N.$ It is clear that $0\leq u^{(1)}\leq1.$  From $r+\theta\leq1$ we get that $v^{(1)}\geq0.$ By conditions $\overline{u}\geq1,$  $0<\beta\leq(c+1)(r+\theta)$ or $\overline{u}<1,$  $0<\beta\leq\Psi(\overline{u})$ it follows that $v^{(n)}$ is a decreasing sequence (according to Lemmas \ref{lemma1} and \ref{lemma2}). Let's consider the function $g(x)=\frac{(2-x)(1+cx^2)}{x}.$ Then $g'(x)=-\frac{2(cx^3-cx^2+1)}{x^2}$ and $g'(x)\leq0$ if and only if $cx^3-cx^2+1\geq0.$ From the last inequality we get $c\leq\frac{1}{x^2-x^3}$ and the function  $\frac{1}{x^2-x^3}$ has the minimum value $\frac{27}{4}$ in $(0,1)$ at the point $\frac{2}{3}.$ Thus, if $c\leq27/4$ then the function $g(x)$ is a decreasing function for all $x\in(0,1).$ According to the definition of $u^{(1)},$ if $v^{(0)}\geq\frac{(1-u^{(0)})(1+c(u^{(0)})^2)}{u^{(0)}}$ then $u^{(1)}\leq u^{(0)}.$ Consequently, we get that $v^{(1)}\leq v^{(0)} \leq  g(u^{(0)})\leq g(u^{(1)})$, i.e., $v^{(1)}\leq g(u^{(1)}),$ hence the set $N$ is an invariant set.

From the condition $r+\theta\leq1$ and $\beta\leq(c+1)(r+\theta)$ (or  $\beta\leq\Psi(\overline{u})<(c+1)(r+\theta)$) we obtain that the fixed point $(1,0)$ is an attracting point, so there exists a neighbourhood $U$ of $(1,0)$, such that the trajectory converges to  $(1,0)$ for all initial point from $U$.  Moreover, in this case there is no positive fixed point (according to Lemmas \ref{lemma1} and \ref{lemma2}). Note that the sequence $v^{(n)}$ is always decreasing, the sequence $u^{(n)}$ is increasing in $\left\{(u,v)\in \mathbb{R}^2: 0\leq u\leq1, \, 0\leq v\leq\frac{(1-u)(1+cu^2)}{u}\right\}$ and decreasing in $\left\{(u,v)\in \mathbb{R}^2: 0\leq u\leq1, \, \frac{(1-u)(1+cu^2)}{u}\leq v\leq\frac{(2-u)(1+cu^2)}{u}\right\}.$ To sum up above discussion, we conclude that the trajectory falls into $U$ after some finite steps and then converges to the fixed point $(1,0)$. The proof is complete.
\end{proof}

\begin{rk} Note that if the initial point is taken from
\[\left\{(u,v)\in \mathbb R^2: 1<u\leq2, \, 0\leq v\leq\frac{(2-u)(1+cu^2)}{u}\right\}
 \]
then after one step the trajectory falls in $N.$ Thus, it is enough to study the dynamics in $N.$
\end{rk}

Denote
\[
N_1=\left\{(u,v)\in \mathbb{R}^2: 0< u< 1, \, 0<v\leq\frac{(1-u)(1+cu^2)}{u}\right\},
 \]

  \[
N_2=\left\{(u,v)\in \mathbb{R}^2: 0< u< u_{-}, \, \frac{(1-u)(1+cu^2)}{u}\leq v\leq\frac{(2-u)(1+cu^2)}{u}\right\}.
\]

\begin{pro}\label{upfp} Let $h=2 $, $r+\theta\leq1,$ $c\leq27/4,$ $\beta>(c+1)(r+\theta)$ (or $\overline{u}<1$ and $\beta\in\{(c+1)(r+\theta),\Psi(\overline{u})\}$) and $\overline{q}(u_{-})<1$. If an initial point is taken from the set
$N_1\cup N_2$ then the trajectory converges to the unique positive fixed point $E_{-}=(u_{-},v_{-}).$
\end{pro}

\begin{proof} If $\beta>(c+1)(r+\theta)$ (or $\overline{u}<1$ and $\beta\in\{(c+1)(r+\theta),\Psi(\overline{u})\}$) the according to Lemmas \ref{lemma1} and \ref{lemma2} we have that there exists a unique positive fixed point $E_{-}=(u_{-},v_{-})$ and from $\overline{q} (u_{-})<1$ follows the attractiveness of $E_{-}.$  So there exists a neighbourhood $U(E_{-})$ of $E_{-}$, such that for any initial point taken from $U(E_{-}),$ the trajectory converges to $E_{-}.$  Note that $u_{-}$ is a unique positive solution in $(0,1)$ of the equation $\beta=\Psi(u)=\frac{(r+\theta u)(1+cu^2)}{u^2}.$ Thus, $\beta<\Psi(u)$ if $u\in(0,u_{-})$ and  $\beta>\Psi(u)$ if $u\in(u_{-},1),$  i.e., the sequence $v^{(n)}$ is decreasing in $\left\{(u,v)\in \mathbb{R}^2: 0< u< u_{-}, \, 0\leq v\leq\frac{(2-u)(1+cu^2)}{u}\right\}$ and it is increasing in $\left\{(u,v)\in \mathbb{R}^2: u_{-}<u<1, \, 0\leq v\leq\frac{(2-u)(1+cu^2)}{u}\right\}.$
It is obvious that the sequence $u^{(n)}$ is increasing in $N_1$  and decreasing in $N_2.$ So the trajectory starting from $N_1\setminus U(E_{-})$ falls into $U(E_{-})$ after some finite steps and then converges to  $E_{-}=(u_{-},v_{-}).$ If the initial point is taken from $N_2\setminus U(E_{-})$ then trajectory falls into $N_1$ after some finite steps and the situation is the same in $N_1.$ (Fig. \ref{fig4}). The proof is complete.
 \end{proof}

\begin{figure}[h!]
  \centering
  \includegraphics[width=9cm]{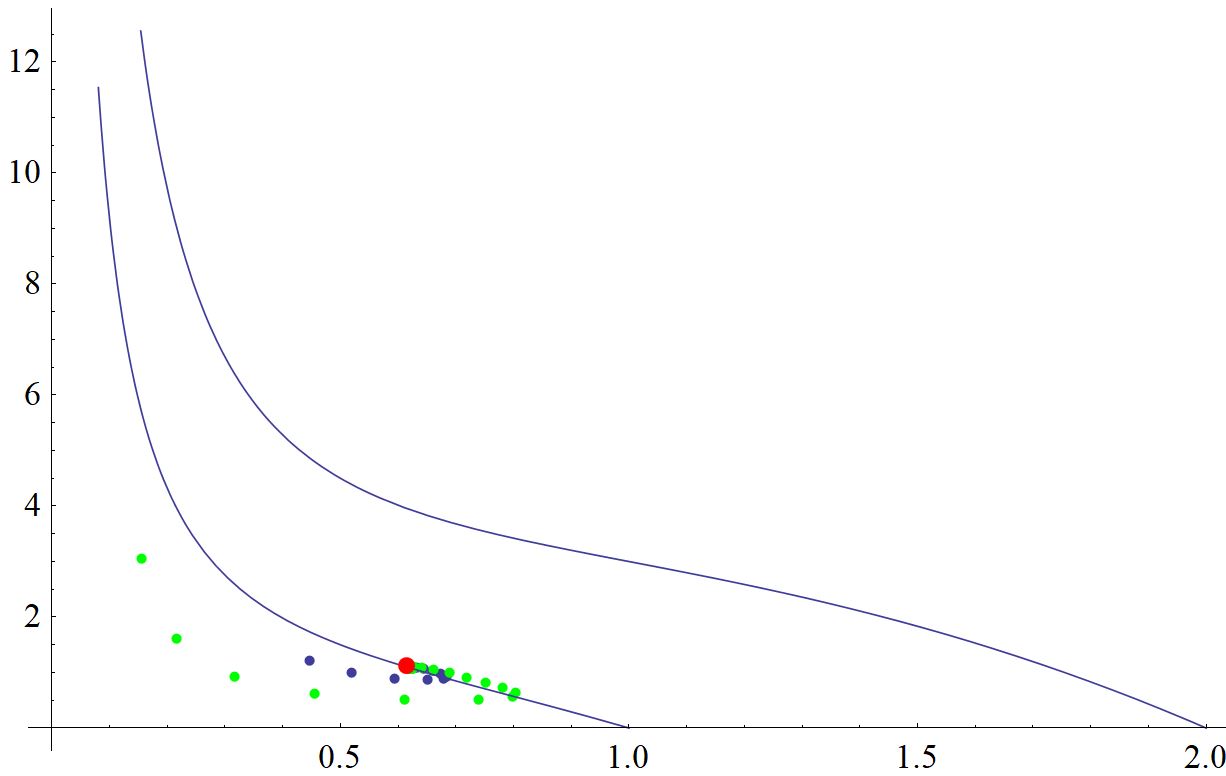}\\
  \caption{$\beta=3,  c=2, r=0.5, \theta=0.25, (u_{-},v_{-})\approx(0.623, 1.074)$. Blue is the trajectory of initial point $u^{(0)}=0.4, v^{(0)}=1.6,$ while green is the trajectory of $u^{(0)}=0.4, v^{(0)}=4.$ }\label{fig4}
\end{figure}

\begin{pro} Let $h=2,$ $r+\theta\leq1,$ $c\leq27/4,$ $\overline{u}<1,$ $\Psi(\overline{u})<\beta<(c+1)(r+\theta)$ and $\overline{q}(u_{-})<1.$ If the initial point is taken from the set $N_1\cup N_2$  then the stable manifold of $(u_{+},v_{+})$ divides the set $N$ into two parts  $\Omega_1$ (which includes $E_{-}$) and $\Omega_2$ such that the trajectory starting from $\Omega_1$ converges to the fixed point $E_{-}=(u_{-},v_{-})$ and the trajectory starting from $\Omega_2$ converges to the fixed point $E_{1}=(1,0).$
\end{pro}

\begin{proof} If $\overline{u}<1$ and $\Psi(\overline{u})<\beta<(c+1)(r+\theta)$ then according to Lemmas \ref{lemma1} and \ref{lemma2}, it follows that there exist two positive fixed points $E_{-}=(u_{-},v_{-})$, $E_{+}=(u_{+},v_{+})$ and from  $\overline{q}(u_{-})<1$ we have that the fixed point $E_{-}=(u_{-},v_{-})$ is an attracting point. Moreover, from  $r+\theta\leq1$  and $\beta<(c+1)(r+\theta)$ we get that the fixed point $(1,0)$ is also an attracting point. According to Proposition \ref{type}, the fixed point $(u_{+},v_{+})$ is a saddle point. Then for the saddle point $(u_{+},v_{+})$ the curve $\frac{(1-u)(1+cu^2)}{u}$ plays the role of an unstable manifold and there is another stable curve $\gamma$ which divides the set $N$ into two parts (in Fig.\ref{fig5} we tried to find $\gamma$ using a numerical solution). Here also as in Proposition \ref{upfp}, first the trajectory starting from $\Omega_1$ (resp. $\Omega_2$) approaches to the curve $\frac{(1-x)(1+cx^2)}{x},$ and then it goes to the fixed point $E_{-}$ (resp. $E_{1}$).
\end{proof}

\begin{figure}[h!]
  \centering
  \includegraphics[width=9cm]{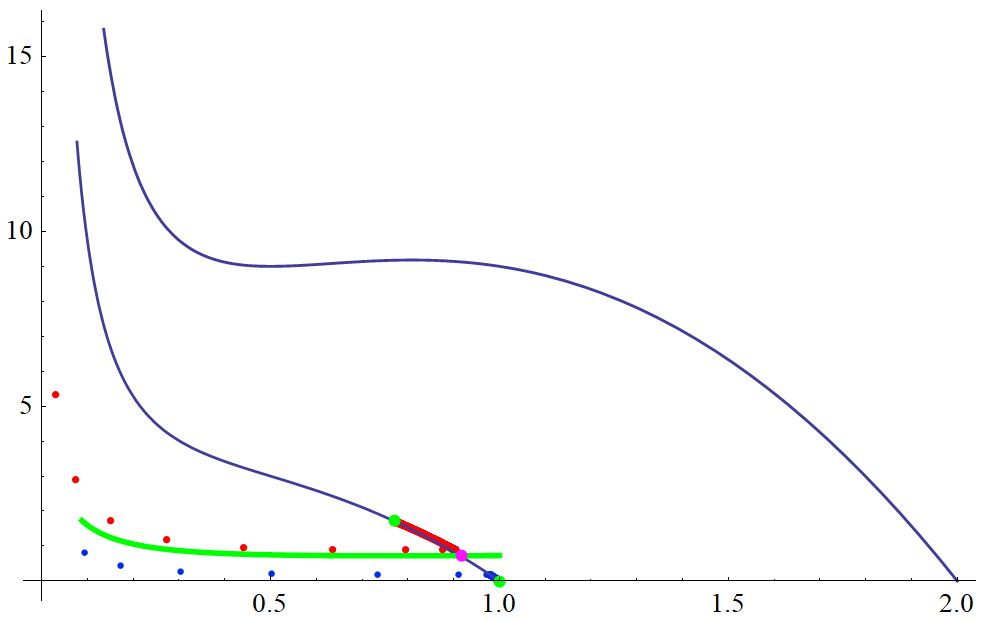}\\
  \caption{$\beta=6.7,  c=8, r=0.5, \theta=0.25, (u_{-},v_{-})\approx(0.784712, 1.625865), (u_{+},v_{+})\approx(0.913894, 0.723753)$. Red is the trajectory of initial point $u^{(0)}=0.1, v^{(0)}=10,$ while blue is the trajectory of $u^{(0)}=0.05, v^{(0)}=1.6.$ Green line is a stable manifold passing through $(u_{+},v_{+})$. }\label{fig5}
\end{figure}

\begin{rk} If $\overline{q}(u_{-})=1$ then the characteristic equation of Jacobian of the operator (\ref{h2}) has a pair of complex conjugate solutions (By Lemma \ref{keylem}).  In this situation can occur the Neimark-Sacker bifurcation and appears an invariant closed curve about  $(u_{-},v_{-})$  when one of the parameters $\beta, r, \theta, c$ varies in a small neighbourhood of the value of that parameter which satisfies the equation $\overline{q}(u_{-})=1.$ In Fig \ref{fig6} we will consider trajectories with some parameter changes.
\end{rk}

\begin{figure}[h!]
    \centering
    \subfigure[\tiny$c=11.3, u_{-}\approx0.371, \overline{q}(u_{-})\approx0.999, (0.3, 4).$]{\includegraphics[width=0.45\textwidth]{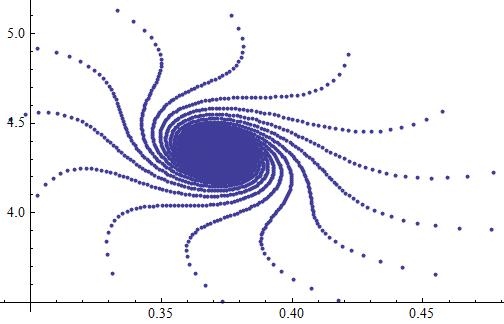}} \hspace{0.3in}
    \subfigure[\tiny$c=11.1, u_{-}\approx0.365, \overline{q}(u_{-})\approx1.003, (0.36, 4.3).$]{\includegraphics[width=0.45\textwidth]{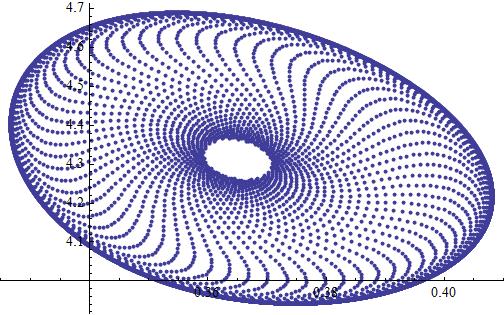}}
    \subfigure[\tiny$c=11.1, u_{-}\approx0.365, \overline{q}(u_{-})\approx1.003, (0.4, 5).$]{\includegraphics[width=0.45\textwidth]{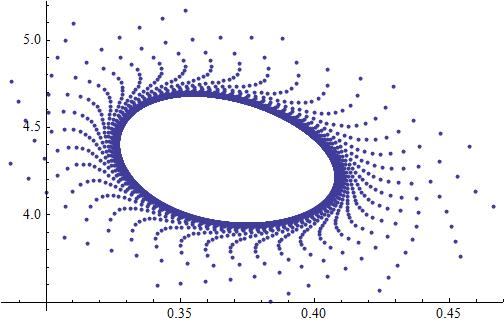}} \hspace{0.3in}
    \subfigure[\tiny Figures (b) and (c) in one]{\includegraphics[width=0.45\textwidth]{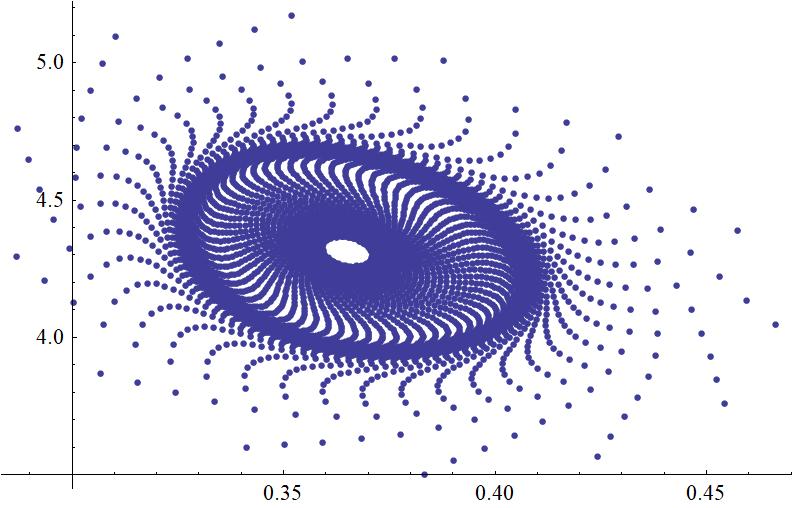}}
    \caption{Phase portraits for the system (\ref{h2}) with $\beta=11, r=0.5, \theta=0.25, n=10000.$}
    \label{fig6}
\end{figure}

\begin{rk} Let $0<r<1$. From the system (\ref{h2}), it can be seen that for the initial point $(u^{(0)},v^{(0)}):$

(1) if $u^{(0)}<0$ and $v^{(0)}>0$ then $u^{(n)}\rightarrow-\infty$ and $v^{(n)}\rightarrow\infty.$   It means that if the initial point is taken from the first quadrant and $v^{(0)}>\frac{(2-u^{(0)})(1+c(u^{(0)})^2)}{u^{(0)}}$, then the the trajectory goes to $\infty;$

(2) if $u^{(0)}>0,$  $v^{(0)}<0$  or $u^{(0)}<0,$  $v^{(0)}<0$ then $u^{(n)}\rightarrow-\infty$ and $v^{(n)}\rightarrow-\infty.$ It means that if $r+\theta>1$ then the trajectory may go to $\infty.$
\end{rk}

\begin{rk} For parameters $r+\theta>1$ or other parameter values satisfying $\overline{q}(u_{-})>1$, the convergence of the trajectory in $N$ depends not only on the parameters but also on the starting point.
\end{rk}

\section{Discussion}

In this paper, we investigated the discrete dynamics of a phytoplankton-zooplankton system characterized by Holling type II ($h=1$) and Holling type III ($h=2$) predator functional responses. In \cite{SH}, the existence and local stability of fixed points were examined for the case of $h=1$, revealing that the system undergoes a Neimark-Sacker bifurcation at one positive fixed point. Here, we explore the global stability of fixed points for both $h=1$ and $h=2$.

For the system with $h=2$, we analyze the type of  positive fixed points. The fixed point $E_{+}$ is identified as a saddle point if it exists, while the fixed point $E_{-}$ is attracting if $\overline{q}(u_{-})<1$. Note that this condition differs from that in the continuous case described in \cite{Chen}, where it is stated as $g'(u_{-})<0$. Additionally, it is remarked that, for the case of $h=2$ also a Neimark-Sacker bifurcation occurs at the fixed point $E_{-}$.

We note that if there exist two positive fixed points ($\theta\neq0$), the local stability of $E_{-}$ implies that the model occurs bistable phenomenon (as in \cite{Chen}), meaning toxin effects can induce bistable phenomena.

\section*{Acknowledgements}
The first author thanks to the International Mathematical Union (IMU-Simons Research Fellowship Program) for
providing financial support of his visit to the Paul Valery University, Montpellier, France.


\begin{thebibliography}{99}

\bibitem{Chatt}
J. Chattopadhayay, R. R. Sarkar and S. Mandal, {\em Toxin-producing plankton may act as a biological control for planktonic blooms-Field study and mathematical modelling}, J. Theor. Biol., 2002, 215(3), 333--344.

\bibitem{Chen2}
J. Chen and H. Zhang, {\em The qualitative analysis of two species predator-prey model with Holling type III functional response}, Appl. Math. Mech., 1986, 77(1), 77--86.

\bibitem{Chen}
S. Chen, H. Yang and J. Wei, {\em Global dynamics of two phytoplankton-zooplankton models with toxic substances effect}, Journal of Applied Analysis and Computation, 2019, 9(3), 796--809.

\bibitem{Cheng}
W. Cheng and L. Wang, {\em Stability and Neimark-Sacker bifurcation of a semi-discrete population model}, Journal of Applied Analysis and Computation, 2014, 4(4), 419--435.

\bibitem{De}
R. L. Devaney, {\em An Introduction to Chaotic Dynamical System}, Westview Press, 2003.

\bibitem{Hong}
Y. Hong, {\em Global dynamics of a diffusive phytoplankton-zooplankton model with toxic substances effect and delay}, Math. Biosci. Eng., 2022, 19(7), 6712--6730.

\bibitem{Hsu1}
S. B. Hsu, {\em On global stability of a predator-prey system}, Math. Biosci., 1978, 39(1--2), 1--10.


\bibitem{Hsu3}
S. B. Hsu and T. Huang, {\em Global stability for a class of predator-prey systems}, SIAM J. Appl. Math., 1995, 55(3), 763--863.

\bibitem{Ko}
W. Ko and K. Ryu, {\em Qualitative analysis of a predator-prey model with Holling type II functional response incorporating a prey refuge}, J. Differential Equations, 2006, 231(2), 534--550.


\bibitem{Tian}
T. Liao, {\em The impact of plankton body size on phytoplankton-zooplankton dynamics in the absence and presence of stochastic environmental fluctuation}, Chaos, Solitons Fractals, 2022. DOI: 10.1016/j.chaos.2021.111617.

\bibitem{Peng}
R. Peng and J. Shi, {\em Non-existence of non-constant positive steady states of two Holling type-II predator-prey systems: Strong interaction case}, J. Differential Equations, 2009, 247(3), 866--886.


\bibitem{RSH}
U. A. Rozikov and  S. K. Shoyimardonov, {\em Ocean ecosystem discrete time dynamics generated by $\ell$-Volterra operators}, International Journal of Biomathematics, 2019, 12(2), 1950015--1--24.

\bibitem{RSHV}
U. A. Rozikov, S. K. Shoyimardonov and R. Varro, {\em Planktons discrete-time dynamical systems}, Nonlinear studies, 2021, 28(2), 585--600.

\bibitem{SH}
S. Shoyimardonov, {\em Neimark-Sacker bifurcation and stability analysis in a discrete phytoplankton-zooplankton system with Holling type II functional response}, Journal of Applied Analysis and Computation, 2023, 13(4), 2048--2064.

\bibitem{Wang}
J. Wang, {\em Spatiotemporal patterns of a homogeneous diffusive predator-prey system with Holling type III functional response}, J. Dyn. Diff. Equat., 2017, 29(4), 1383--1409.

\bibitem{Wing}
S. Winggins, {\em Introduction to Applied Nonlinear Dynamical Systems and Chaos}, Springer-Verlag, New York, 2003.

\bibitem{Zhou}
J. Zhou and C. Mu, {\em Coexistence states of a Holling type-II predator-prey system}, J. Math. Anal. Appl., 2010, 369(2), 555--563.

\bibitem{Qiu}
Q. Zhao, S. Liu and X. Niu, {\em Dynamic behavior analysis of a diffusive plankton model with defensive and offensive effects}, Chaos, Solitons Fractals, 2019, 129, 94--102.

\end{thebibliography}

\end{document}